\documentclass[12pt,oneside,reqno]{amsart}
\usepackage[utf8]{inputenc}
\usepackage[english]{babel}
\usepackage{amsmath}
\usepackage{amsfonts}
\usepackage{amssymb}
\usepackage{graphicx}
\usepackage{setspace}
\usepackage{parskip}
\usepackage{amsthm}
\usepackage{fancyhdr}
\usepackage{booktabs}
\usepackage{mathabx}
\usepackage{hyperref}
\usepackage{pgfplots}
\pgfplotsset{compat=newest}
\usepackage{sidecap}
\usepackage{tikz}
\usepackage{tikz-3dplot}
\usepackage{xcolor}
\usepackage{xstring}
\usepackage{circuitikz}
\usepackage{caption}
\usepackage{subcaption}
\usepackage{listings}
\usepackage{array,amsfonts}
\usetikzlibrary{plotmarks}
\usepackage[left=2.2cm,right=2.2cm,top=3cm,bottom=2.5cm]{geometry}
\usepackage{comment}
\usepackage[noabbrev]{cleveref}

\newcommand{\separator}{
  \begin{center}
    \rule{\columnwidth}{0.3mm}
  \end{center}
}

\newcommand{\beq}{\begin{eqnarray*}}
\newcommand{\eeq}{\end{eqnarray*}}
\newcommand{\beqn}{\begin{eqnarray}}
\newcommand{\eeqn}{\end{eqnarray}}
\newcommand{\bemn}{\begin{multiline}}
\newcommand{\eemn}{\end{multiline}}

\def\N{\mathbb{N}}

\def\Q{\mathbb{Q}}
\def\Z{\mathbb{Z}}
\def\C{\mathbb{C}}

\newtheorem{teorema}{\bf Theorem}

\newtheorem{lema}{\bf Lemma}

\newtheorem{proposicion}{\bf Proposition}
\newtheorem{conjetura}{\bf Conjecture}

\numberwithin{teorema}{section}
\numberwithin{definicion}{section}
\numberwithin{lema}{section}
\numberwithin{conjetura}{section}
\numberwithin{proposicion}{section}

\def\sl{\mathfrak{sl}}
\def\Uh{\mathcal{U}_h}
\def\Uq{U_q}
\def\g{\mathfrak{g}}

\def\id{{\mathrm{id}}}

\def\okef{\overline{k_e}_{{}_1}}
\def\okes{\overline{k_e}_{{}_2}}
\def\okff{\overline{k_f}_{{}_1}}
\def\okfs{\overline{k_f}_{{}_2}}

\def\kfs{k_{f_2}}
\def\ketf{k_{e_{31}}}
\def\kets{k_{e_{32}}}
\def\kftf{k_{f_{31}}}
\def\kfts{k_{f_{32}}}
\def\onf{\overline{n_1}}
\def\ons{\overline{n_2}}

\newcommand{\qbinom}[2]{\left[\genfrac{}{}{0pt}{}{#1}{#2}\right]_q}

\newcommand{\cL}{\mathcal{L}}
\begin{document}

\title[A large color $R$-matrix for $\sl_3$]{A large color $R$-matrix for $\sl_3$}
\author{Angus Gruen and Lara San Mart\'in Su\'arez}
\date{}
\begin{abstract}
    We construct the invariant $F_K^{\sl_3}\in\mathbb{Z}[q,q^{-1}][[x,y]]$ for any positive braid knot $K$, whose existence was conjectured by Park, building on earlier work of Gukov--Manolescu. The main step in our work extends a result by the first author on the invariant associated to symmetric representations of $\sl_3$ to all irreducible representations. We conclude with a conjectural framework for constructing $F_K^{\sl_N}$ for arbitrary $N$.
\end{abstract}
\maketitle

\section{Introduction}

Ever since the introduction of Khovanov Homology \cite{Kh} which categorified the Jones polynomial, categorification of complex Chern--Simons theory has been one of the largest open problems in quantum topology. In a series of two papers, Gukov, Putrov and Vafa \cite{GPV} and Gukov, Pei, Putrov and Vafa \cite{GPPV} approach this open problem by introducing the $q$-series valued invariant of 3-manifolds $\hat{Z}$, such that the radial limits at $q$ root of unity recover the Witten--Reshetikhin--Turaev (WRT) invariants \cite{Witten,RT}, and conjectured the existence of a homology theory categorifying $\hat{Z}$. This invariant was initially defined for negative plumbed manifolds, and extending this invariant to arbitrary 3-manifolds is a current open research problem.

In an effort to generalize the definition of $\hat{Z}$, Gukov and Manolescu \cite{GM} introduced the knot invariant $F_K:=\hat{Z}(S^3\setminus K)$. This knot counterpart of $\hat{Z}$ is now a power series in two variables, $x$ and $q$, and can be related to $\hat{Z}$ for a large family of manifolds via Dehn surgery. Additionally, Gukov and Manolescu conjecture that $F_K$ can be obtained as a resummation of the Melvin--Morton--Rozansky expansion of the colored Jones polynomial.

\begin{conjetura}[Conjecture 1.5 in \cite{GM}] For every knot $K$ in $S^3$, there exists a two-variable series $$F_K(x,q)=\sum_{m=0}^\infty f_m(q)\, x^{m+\frac12}$$ with $f_m\in\mathbb{Z}[[q]][q^{-1}]$ that satisfies
\begin{equation*}
    F_K(x,e^\frac{h}{2})\sim\frac{1}{\Delta_K(x)}+\sum_{k=1}^\infty \frac{P_k(x)}{\Delta^{2k+1}_K(x)}\frac{h^k}{k!}
\end{equation*}
where the RHS is the Melvin--Morton--Rozansky expansion of the colored Jones polynomial \cite{MM,Roz,BNG} and $\sim$ indicates that both sides are related via Borel resummation.

Moreover, this series is annihilated by the quantum $A$-polynomial \cite{Gar,Guk05}
\begin{equation*}
    \hat{A}_k(\hat{x},\hat{y},q)F_K(x,q)=0
\end{equation*}
\end{conjetura}

In a series of papers, Park proved this conjecture first for positive braid knots \cite{ParkLargeColor}, and later for knots admitting a \textit{signed braid diagram} \cite{Park21}, by finding a description of $F_K$ in terms of suitably modified $R$-matrices.

As expected from its relation to $SU(2)$ complex Chern--Simons theory and the colored Jones polynomial, the invariant $F_K$ is associated with the choice of the Lie algebra $\sl_2$. In \cite{Park20}, Park extends the definition of $\hat{Z}$ and $F_K$ to any semisimple Lie algebra $\g$ and gives an explicit formula of $\hat{Z}^\g$ for negative plumbed manifolds and $F_K^\g$ for torus knots.

As we move outside of $\sl_2$ to any semisimple Lie algebra $\g$, the description of the $F_K$ invariant becomes more involved. The relation between $F_K$ and the colored Jones polynomial suggest that each element in the dominant integral weight basis of $\g$ corresponds to a new variable dependence of $F_K^\g$. In the case of $\g=\sl_2$, the dominant integral weights are in bijection with positive integers, which is why $F_K$ depends on the additional variable $x$. More generally in $\sl_N$, the dominant integral weights are in bijection with $N-1$ positive integers and so $F_K^{\sl_N}$ depends on $N$ variables, $x_1,\dots,x_{N-1}$ and $q$. The structure of $F_K^\g$ for other semisimple Lie algebra $\g$ can be argued analogously. 

In order to generalize the efforts made to define $F_K$ to $F_K^\g$, the simplest case was studied first. This gives rise to the definition of $F_K^{\sl_N,sym}$, which can be thought as a ``large color limit"\footnote{Here ``color" refers to the irreducible representation which we assign to each connected component of a knot diagram.} of the Jones polynomial associated to the tensor products of fundamental representations. This invariant was initially introduced in \cite{EGGKPS,Park20} as the restriction of $F_K^{\sl_N}$ to a particular one-dimensional subspace in the weight lattice\footnote{This can be achieved algebraically by setting all but one of the additional variables in $F_K^{\sl_N}$ to $1$.}, for which the first author later gave an $R$-matrix definition \cite{gruen22}, by constructing the associated Verma module in $\Uq(\sl_N)$ that gives rise to $F_K^{\sl_N,sym}$.

One major drawback of $F_K^{\sl_N,sym}$ is that, unlike $F_K^{\sl_N}$, there is no surgery formula relating $F_K^{\sl_N,sym}$ to higher rank $\hat{Z}$. Hence, these $F_K^{\sl_N,sym}$ should be thought of as a stepping stone on the way to defining the full invariant. The primary result of this paper is a construction of the full invariant in $\sl_3$ for positive braid knots, which proves Theorem \ref{thm:FK-SU(3)}, along with a conjecture for the construction of the full invariant in $\sl_N$.

\begin{teorema}[$F_K$ for $\sl_3$]\label{thm:FK-SU(3)} Let $\beta_K$  be a positive braid whose closure is the knot $K$. Then, the reduced quantum trace $\widetilde{\text{Tr}}^q_{V_x^{1}\otimes V_y^{2}}(\beta_K)$ converges in $\mathbb{Z}[q,q^{-1}][[x^{-1},y^{-1}]]$ and $$F_K^{\sl_3}(x,y,q)=\sum_{i,j\geq 0} f_{i,j}(q)\,x^{i+\frac12}y^{j+\frac12}$$ is a well-defined knot invariant that satisfies
    \begin{gather}\label{eq:thm-MMR}
        F_K^{\sl_3}(x,y,e^{\frac{h}{2}})=\sum_{j=0}^\infty \frac{P_j(K;x,y)}{\left(\Delta_K(x)\Delta_K(y)\Delta_K\left((xy)^{-1}\right)\right)^{2j+1}}\frac{h^j}{j!},
    \end{gather}
    where the R.H.S. is the perturbative expansion of the $\sl_3$ quantum knot invariant, such that $P_j(K;x,y)\in\Q[x,x^{-1},y,y^{-1}]$ and $P_0=1$. The equality holds after expanding $x$, $y$ and $h$ near zero.
    
    In particular, in the classical limit it recovers the Alexander polynomial
    \begin{equation}\label{eq:thm-alexander}
        \lim_{q\rightarrow 1} F_K^{\sl_3}(x,y,q)=\frac{1}{\Delta_K(x)\Delta_K(y)\Delta_K\left((xy)^{-1}\right)}.
    \end{equation}
\end{teorema}

In order to prove Theorem \ref{thm:FK-SU(3)}, we derive a way to explicitly compute the quantum knot invariant associated to any finite-dimensional irreducible $\sl_3$ representation for any knot. Previous approaches to compute this invariant have involved cabling \cite{cabling} or certain combinatorial descriptions \cite{2-webs}. Our derivation also makes use of cabling but differs in the method by which to restrict to the target representation. Instead of using projectors, we take the reduced quantum trace\footnote{For a definition of the reduced quantum trace, see Figure \ref{fig:reduced-QT}.} of a braid diagram using appropriate coloring on the open strand. This approach restricts us to only working with knots (instead of more general braids); however, this is sufficient for our purposes.

In addition to providing an explicit $R$-matrix construction of the $F_K^{\sl_3}$ invariant, we also conjecture a framework for how to construct $F_K^{\sl_N}$ for general $N$, answering one of the questions posed in \cite{ParkLargeColor}. We believe that, in this framework, it should also be possible to extend some of the additional techniques from \cite{ParkLargeColor,Park21,gruen22} (such as the stratified and inverted state sums) to the $\sl_N$ case but we leave this for a future work.

As we will discuss later, one of the main difficulties of our approach is computational. The growth in the number of internal variables and summations becomes unsustainable for the computations of knots with a large number of crossings, as is usual in constructions defining knot invariants via cabling. For this reason, we only compute the trefoil and cinquefoil explicitly. However, as computational power grows, we believe that it would be possible to implement the computation of $F_K^{\sl_N}(x_1,\dots,x_{N-1},q)$ for any positive braid knot using this procedure.

\subsection*{Paper outline}

The paper is structured as follows. In Section \ref{sec:prerequisites} we review the necessary ingredients for constructing the quantum knot invariant associated with the $\sl_3$ Lie algebra. This part lays the groundwork for the study and establishes the conventions adopted throughout the paper.

Building on this foundation, in Section \ref{sec:construction} we revisit the Verma modules introduced by the first author in \cite{gruen22} and detail our construction, which centers on the action of the $R$-matrix on their tensor product. We provide an explicit formula for this $R$-matrix, study its symmetric limit, and use this formula to prove Theorem \ref{thm:FK-SU(3)}.

Next, in Section \ref{sec:examples} we demonstrate how to construct explicit examples for torus knots using the $R$-matrix formula and compare these results with those obtained by Park in \cite{Park20}.  Lastly, in Section \ref{sec:remarks-higher-rank} we discuss the extension of this construction to $\sl_N$ and examine the main challenges involved in such generalizations.

\subsection*{Conventions} 

Here we define the conventions we use for the quantum integer, quantum integer factorial and quantum binomial.
\begin{gather*}
    [n]_q=\frac{q^{\frac{n}{2}}-q^{-\frac{n}{2}}}{q^{\frac{1}{2}}-q^{-\frac{1}{2}}}, \hspace{2em} [k]_{x,q}=\left.[n-k]_q\right|_{q^n\rightarrow x}=\frac{x^\frac12q^{-\frac{k}{2}}-x^{-\frac12}q^\frac{k}{2}}{q^\frac12-q^{-\frac12}} \\
    \left(a;q\right){}_{n} = \prod_{i=0}^{n-1} (1-aq^i), \hspace{2em} \left(q\right){}_{n} := \left(q;q\right){}_{n} = \prod_{i=1}^n (1-q^i) \\
    [n]_q!:=\prod_{i=1}^n [i]_q, \hspace{2em} \qbinom{n}{k}:=\frac{\left(q\right){}_{n}}{\left(q\right){}_{k}\left(q\right){}_{n-k}} = q^{\frac{(n-k)k}{2}}\frac{[n]_q!}{[k]_q![n-k]_q!} \\
    \qbinom{n}{k_1\;k_2\;k_3}:=\frac{\left(q\right){}_{n}}{\left(q\right){}_{k_1}\left(q\right){}_{k_2}\left(q\right){}_{k_3}}
\end{gather*}

\subsection*{Acknowledgments}

L.S.M.S. acknowledges financial support by the mobility grants program of Centre de Formació Interdisciplinària Superior (CFIS) -- Universitat Politècnica de Catalunya (UPC).

\section{Prerequisites}\label{sec:prerequisites}

\subsection{Representations of \texorpdfstring{$\sl_3$}{sl3}}

We will denote by $\{\alpha_1,\alpha_2\}$ the set of simple roots and $\{\omega_1,\omega_2\}$ the fundamental weights. The Cartan matrix $C$ of $\sl_3$ takes the form
\begin{equation*}
    C = \begin{pmatrix}
    2 & -1 & -1\\
    -1 & 2 & -1\\
    -1 & -1 & 2
    \end{pmatrix}.
\end{equation*}

The Lie algebra $\sl_3=\sl(3,\C)$ can then be generated by the two $\sl_2$ triples $\{H_i,E_i,F_i\}_{i=1,2}$, satisfying the usual relations and
\begin{gather}
    \label{eq:sl3-H-relation} [H_i,H_j]=0 \\
    \label{eq:sl3-HE-relation1} [H_1,E_1]=2\,E_1, \hspace{0.5em} [H_1,E_2]=-E_2, \hspace{0.5em}
    [H_1,F_1]=-2\,F_1, \hspace{0.5em} [H_1,F_2]=F_2 \\
    \label{eq:sl3-HF-relation2} [H_2,E_1]=-E_1, \hspace{0.5em} [H_2,E_2]=2\,E_2, \hspace{0.5em}
    [H_2,F_1]=F_1, \hspace{0.5em} [H_2,F_2]=-2\,F_2 \\
    [E_i,F_j]=\delta_{ij}\,H_i \\ \label{eq:sl3-E-relation}
    \sum_{k=0}^2 (-1)^k \binom{2}{k} E_1^kE_2E_1^{2-k}=\sum_{k=0}^2 (-1)^k \binom{2}{k} E_2^kE_1E_2^{2-k}=0 \\
    \label{eq:sl3-F-relation} \sum_{k=0}^2 (-1)^k \binom{2}{k} F_1^kF_2F_1^{2-k}=\sum_{k=0}^2 (-1)^k \binom{2}{k} F_2^kF_1F_2^{2-k}=0 
\end{gather}

Given a representation of $\sl_3$, we say that an ordered pair $\mu=(n_1,n_2)$ is a weight for such representation if there exists a vector which is simultaneously a $H_i$-eigenvector of eigenvalue $n_i$ for $i=1,2$. Given two weights $\mu_1$ and $\mu_2$ of $\sl_3$, we say that $\mu_1$ is higher than $\mu_2$ if there exists non-negative integers $a,b$ such that $\mu_1-\mu_2=a\alpha_1+b\alpha_2$. A representation of $\sl_3$ is said to be of highest weight $\mu$ if for any other weight $\lambda$ for the representation, $\mu$ is higher than $\lambda$. There is a one-to-one correspondence between finite-dimensional irreducible representations of $\sl_3$ and pairs of non-negative integers $n,m\in\N^+$, such that $\mu=(n,m)$ is the highest weight associated to the representations.

Representations with $m=0$ (resp. $n=0$) are referred to as symmetric of degree $n$ (resp. $m$), and these arise as a particular example of the following infinite--dimensional representation. Define the map
\begin{equation}\label{eq:Ei-Fi-Hi}
    E_i \mapsto z_i\frac{\partial}{\partial z_{i+1}} \hspace{2em} F_i \mapsto z_{i+1}\frac{\partial}{\partial z_{i}} \hspace{2em} H_i \mapsto z_i\frac{\partial}{\partial z_{i}} - z_{i+1}\frac{\partial}{\partial z_{i+1}}
\end{equation}
It induces an action of $\sl_3$ on $\C\left[z_1,z_2,z_3\right]$ which decomposes as the direct sum of eigenspaces
\begin{equation*}
    \C\left[z_1,z_2,z_3\right] = \bigoplus_{n=0}^\infty V_{(n,0)},
\end{equation*}
where $V_{(n,m)}$ represents the $\sl_3$ irreducible representation of highest weight $(n,m)=n\,\omega_1+m\,\omega_2$. 

If the roles of the triples $\{H_1,E_1,F_1\}$ and $\{H_2,E_2,F_2\}$ are exchanged in \eqref{eq:Ei-Fi-Hi}, there is another map which gives rise to an action of $\sl_3$ on $\C\left[z_1,z_2,z_3\right]$ given by
\begin{alignat}{2}
    &E_1 \mapsto z_2\frac{\partial}{\partial z_3}, \hspace{3em} &&E_2\mapsto z_1\frac{\partial}{\partial z_2} \\
    &F_1 \mapsto z_3\frac{\partial}{\partial z_2}, \hspace{3em} &&F_2 \mapsto z_2\frac{\partial}{\partial z_1}\\
    &H_1 \mapsto z_2\frac{\partial}{\partial z_2}-z_3\frac{\partial}{\partial z_3}, \hspace{3.5em} &&H_2\mapsto z_1\frac{\partial}{\partial z_1}-z_2\frac{\partial}{\partial z_2}
\end{alignat}
such that $\C\left[z_1,z_2,z_3\right]$ decomposes instead as
\begin{equation*}
    \C\left[z_1,z_2,z_3\right] = \bigoplus_{m=0}^\infty V_{(0,m)}.
\end{equation*}
The symmetric representations $V_{(n,0)}$ and $V_{(0,m)}$ correspond under the previous action to the $\C$--eigenspace generated by the monomials of fixed degree $n$ and $m$ respectively.

Given two $\sl_3$ representations $V$ and $W$, there is an $\sl_3$ action on the tensor product $V\otimes W$ for which each element $X\in \sl_3$ acts by $X\otimes I + I\otimes X$.

\begin{lema}[Littlewood--Richardson rule, \cite{Littlewood-Richardson}]\label{lemma:LR-rule} Let $V_{(n,m)}$ denote the $\sl_3$ irreducible representation with highest weight $(n,m)$. Then,
    \begin{equation}\label{eq:tensorprod-decomposition-classical}
        V_{(n,0)}\otimes V_{(0,m)} \cong V_{(n,m)} \oplus \left(V_{(n-1,0)}\otimes V_{(0,m-1)} \right)
\end{equation}
\end{lema}

As a result, the irreducible representation $V_{(n,m)}$ can be constructed as the subrepresentation of $V_{(n,0)}\otimes V_{(0,m)}$ generated by the vector with weight $(n,m)$.

\subsection{Representations of \texorpdfstring{$\Uq(\sl_3)$}{Uqsl3}}

The quantum enveloping algebra $\Uq(\sl_3)$ is the Hopf algebra generated by the two triples $\{K_i,E_i,F_i\}_{i=1,2}$ satisfying the relations
\begin{gather}
    \label{eq:Uhsl3-K-relation} K_iK_j=K_jK_i, \qquad K_iK_i^{-1}=K_i^{-1}K_i=1 \\
    \label{eq:Uhsl3-KE-relation1} K_1E_1=qE_1K_1, \hspace{0.5em} K_1E_2=q^{-\frac12}E_2K_1, \hspace{0.5em}
    K_1F_1=q^{-1}F_1K_1, \hspace{0.5em} K_1F_2=q^\frac12F_2K_1 \\
    \label{eq:Uhsl3-KE-relation2} K_2E_1=q^{-\frac12}E_1K_2, \hspace{0.5em} K_2E_2=qE_2K_2, \hspace{0.5em}
    K_2F_1=q^\frac12F_1K_2, \hspace{0.5em} K_2F_2=q^{-1}F_2K_2 \\
    [E_i,F_j]=\delta_{ij}\frac{K_i-K_i^{-1}}{q^\frac12-q^{-\frac12}} \\ \label{eq:Uhsl3-E-relation}
    \sum_{k=0}^2 (-1)^k \qbinom{2}{k} E_1^kE_2E_1^{2-k}=\sum_{k=0}^2 (-1)^k \qbinom{2}{k} E_2^kE_1E_2^{2-k}=0 \\
    \label{eq:Uhsl3-F-relation} \sum_{k=0}^2 (-1)^k \qbinom{2}{k} F_1^kF_2F_1^{2-k}=\sum_{k=0}^2 (-1)^k \qbinom{2}{k} F_2^kF_1F_2^{2-k}=0 
\end{gather}
The comultiplication $\Delta$, counit $\varepsilon$ and antipode $S$ are
\begin{alignat}{6}
    \label{eq:Uhsl3-structure-1} &\Delta(K_i) &&= K_i\otimes K_i, \hspace{1em} &&\varepsilon(K_i) &&= 1, \hspace{1em} &&S(K_i)&&=K_i^{-1} \\
    &\Delta(E_i) &&= E_i\otimes K_i + 1\otimes E_i, \hspace{1em} &&\varepsilon(E_i) &&= 0, \hspace{1em} &&S(E_i)&&=-E_iK_i^{-1} \\
    \label{eq:Uhsl3-structure-3} &\Delta(F_i) &&= F_i\otimes 1 + K_i^{-1}\otimes F_i, \hspace{2em} &&\varepsilon(F_i) &&= 0, \hspace{2em} &&S(F_i)&&=-K_iF_i
\end{alignat}

Let $\Uh(\sl_3)$ denote the ribbon Hopf algebra generated by $\{H_i,E_i,F_i\}_{i=1,2}$ under the substitution $q=e^{h}$ and $K_i=q^{\frac{H_i}{2}}=e^{\frac{h H_i}{2}}$. Note that $\Uq(\sl_3)$ is the subalgebra of $\Uh(\sl_3)$ generated by $\{q^{\frac{H_i}{2}},E_i,F_i\}_{i=1,2}$.

The universal $R$-matrix of $\Uh(\sl_3)$ was derived in \cite{Bur} and takes the form

\begin{equation}\label{eq:Bur-formula}
    R_{\sl_3} = q^{\sum_{i,j=1}^2 a_{ij}^{-1} H_i\otimes H_j} \prod_{\alpha\in\Delta_+} \left(\sum_{n_\alpha=1}^\infty \frac{q^\frac{n_\alpha(n_\alpha-1)}{4}(1-q^{-1})^{n_\alpha}}{[n_\alpha]_q!} \left(K_\alpha^\frac{1}{2}E_\alpha\otimes K_\alpha^{-\frac{1}{2}}F_\alpha\right)^{n_\alpha} \right),
\end{equation}
where $\Delta_+=\{\alpha_1,\alpha_2,\alpha_3:=\alpha_1+\alpha_2\}$ is the set of positive roots and we are identifying the generators $\{K_i,E_i,F_i\}$ with $\{K_{\alpha_i},E_{\alpha_i},F_{\alpha_i}\}$. The generators associated to the third root $\alpha_3$ are defined as
\begin{equation*}
    E_3:=q^{1/2}E_1E_2-q^{-1/2}E_2E_1, \hspace{2em} F_3:=q^{1/2}F_1F_2-q^{-1/2}F_2F_1.
\end{equation*}

Note that $R_{\sl_3}$ is an element of the completion of $\Uh(\sl_3)\,\otimes\,\Uh(\sl_3)$. Nonetheless, the action of $R_{\sl_3}$ on a finite-dimensional $\Uq(\sl_3)$ module is well-defined and induces a braiding on its category of modules. Since we work exclusively with $\Uq(\sl_3)$ representations, we will use the term \textit{module} and \textit{representation} interchangeably throughout the paper.

The representation theory of $\Uq(\sl_3)$ is parallel to that of $\sl_3$. The finite-dimensional $\Uq(\sl_3)$ representations over $\C(q^\frac16)$ are classified by their highest-weight $\widetilde{\mu}$, which takes the form $\widetilde{\mu} = (\varepsilon_1 q^{n_1},\varepsilon_2 q^{n_2})$ with $\varepsilon_i=\pm 1$, $n_i\in\N^+$ for $i=1,2$. 
If we denote by $V^q_{(n_1,n_2,\varepsilon_1,\varepsilon_2)}$ the finite-dimensional irreducible representation of highest weight $(\varepsilon_1 q^{n_1},\varepsilon_2 q^{n_2})$, the relation $$V^q_{(n_1,n_2,1,1)}\otimes V^q_{(0,0,\varepsilon_1,\varepsilon_2)} \cong V^q_{(n_1,n_2,\varepsilon_1,\varepsilon_2)}$$ allows to reduce the study of all finite-dimensional irreducibles to those of the form $V^q_{(n_1,n_2, 1, 1)}$. From now on, we will denote by $V^q_{(n_1,n_2)}$ the highest-weight finite-dimensional representation $V^q_{(n_1,n_2, 1, 1)}$.

\subsubsection{Quantization of the symmetric representations}

Following \cite{gruen22}, the symmetric representations of $\Uh(\sl_3)$ can be described by the quantization of the $\sl_3$ symmetric representations. Define 
\begin{equation*}
    \left(\frac{\partial f}{\partial z}\right)_q := \frac{f\left(q^\frac{1}{2}z\right)-f\left(q^{-\frac{1}{2}}z\right)}{q^\frac{1}{2}z-q^{-\frac{1}{2}}z}.
\end{equation*}
to be a suitably normalized $q$--derivative of the function $f$. Then, anagolously to the classical case, the map 
\begin{equation}\label{eq:qEi-Fi-Hi}
    E_i \mapsto z_i\left(\frac{\partial}{\partial z_{i+1}}\right)_q \hspace{1.5em} F_i \mapsto z_{i+1}\left(\frac{\partial}{\partial z_{i}}\right)_q \hspace{1.5em} H_i \mapsto z_i\left(\frac{\partial}{\partial z_{i}}\right)_q - z_{i+1}\left(\frac{\partial}{\partial z_{i+1}}\right)_q,
\end{equation}
with $K_i = q^{\frac{H_i}{2}}$, induces an action of $\Uq(\sl_3)$ on $\C(q^{\frac{1}{2}})[z_1,z_2,z_3]$ which decomposes as
\begin{equation*}
    \C(q^{\frac{1}{2}})[z_1,z_2,z_3] = \bigoplus_{n=0}^{\infty} V^q_{(n,0)}.
\end{equation*}
As before, interchanging the labels of the generators decomposes $\C(q^{\frac{1}{2}})[z_1,z_2,z_3]$ in the direct sum of representations of the form $V^q_{(0,m)}$.

Given two $\Uq(\sl_3)$ representations $V$ and $W$, the induced $\Uq(\sl_3)$ action on the tensor product $V\otimes W$ is given by the comultiplication map. The correspondence between the decomposition of the tensor product of $\sl_3$ and $\Uq(\sl_3)$ representations in simple representations \cite{KRT-book} gives the analog of Lemma \ref{lemma:LR-rule} in the quantized case.

\begin{lema}[Littlewood--Richardson rule for $\Uq(\sl_3)$]\label{lemma:LR-rule-Uqsl3} Let $V^q_{(n,m)}$ denote the $\Uq(\sl_3)$ irreducible representation with highest weight $(q^n,q^m)$. Then,
    \begin{equation}\label{eq:tensorprod-decomposition-quantum}
        V^q_{(n,0)}\otimes V^q_{(0,m)} \cong V^q_{(n,m)} \oplus \left(V^q_{(n-1,0)}\otimes V^q_{(0,m-1)} \right)
\end{equation}
\end{lema}

\subsubsection{Evaluation and coevaluation maps}

Given a representation $V$ of $\Uq(\sl_3)$, there are two natural choices for the dual representation, the right dual $V^*$ and left dual ${}^*V$, both of which are isomorphic as vector spaces to the dual vector space of $V$. Given $v\in V$ and $a\in\Uq(\sl_3)$, the action on $V^*$ and ${}^*V$ is given by
\begin{alignat*}{3}
    &(a\cdot \alpha)(v) &&= \alpha(S(a)\cdot v) \hspace{2em} &\text{on } V^*\\
    &(a\cdot \alpha)(v) &&= \alpha(S^{-1}(a)\cdot v) \hspace{2em} & \text{on } {}^*V 
\end{alignat*}

The map $\phi_V:\alpha(\cdot)\mapsto \alpha(u\,\cdot)$ is an isomorphism between ${}^*V$ and $V^*$, where 
\begin{equation*}
    u = K_1K_2.
\end{equation*}
Define the evaluation and coevaluation maps of $\Uq(\sl_3)$ as
\begin{alignat*}{2}
    \overleftarrow{ev}_V:\,&V^*\otimes V\rightarrow \C \hspace{1em}\text{ and }\hspace{1em} \overrightarrow{coev}_V:\,&&\C \rightarrow V\otimes V^* \\
    &v_i\otimes v^j\mapsto \delta_{ij} && 1\mapsto \sum_i v_i\otimes v^i
\end{alignat*}
and
\begin{alignat*}{2}
    \overrightarrow{ev}_V:\,&V\otimes V^*\rightarrow \C \hspace{1em}\text{ and }\hspace{1em} \overleftarrow{coev}_V:\,&&\C \rightarrow V^*\otimes V \\
    &v^i\otimes v_j\mapsto u_i\,\delta_{ij} && 1\mapsto \sum_i u_i\,v^i\otimes v_i
\end{alignat*}
where $\{v_i\}_i$ is a basis of $\{K_j\}_{j=1,2}$-eigenvectors of the representation $V$ and $\{v^i\}_i$ the associated dual basis. The element $u_i$ is the scalar such that $u\cdot v_i = u_i v_i$.

\subsection{Quantum knot invariants}

Consider the $n$-strand braid group
\begin{equation*}\label{eq:Bn}
    B_n = \left\{ \sigma_1, \dots, \sigma_{n-1} \;\middle|\;
    \begin{alignedat}{2}
        &\sigma_i\sigma_j = \sigma_j\sigma_i \hspace{6em} \text{if } |i-j|>1, \, &&i,j=1,\dots,n-1 \\
        &\sigma_{i}\sigma_{i+1}\sigma_i=\sigma_{i+1}\sigma_{i}\sigma_{i+1} &&i,j=1,\dots,n-1
    \end{alignedat}
    \right\}.
\end{equation*}

Let $V$ be a finite-dimensional irreducible $\Uq(\sl_3)$ representation. Denote by $P$ the flip map of $V\otimes V$ and by $\widetilde{R}$ the action of the $\Uq(\sl_3)$ R-matrix \eqref{eq:Bur-formula} on $V$. Then, the matrix $R = P\widetilde{R}$ solves the quantum Yang-Baxter equation
\begin{equation}\label{eq:QYB}
    R_{23}R_{12}R_{23}=R_{12}R_{23}R_{12},
\end{equation}
where $R_{ij}$ is induced by the action of $\widetilde{R}$ on the $i^{\text{th}}$ and $j^\text{th}$ components of $V\otimes V\otimes V$.
Thus, the matrix $R$ provides representations of the braid group via $\sigma_i\mapsto R_{i,i+1}$, where $\sigma_i$ denotes the braid group $i^\text{th}$ generator as in \eqref{eq:Bn}.

Define the writhe of the braid $\beta_\cL$ as
\begin{equation*}
    \omega(\beta_\cL):=|\beta_{\cL,+}|-|\beta_{\cL,-}|
\end{equation*}
where $|\beta_{\cL,\pm}|$ is the number of positive and negative crossings respectively. The quantum link invariant associated to the oriented link $\cL$ is
\begin{equation*}
    p_V(\cL;q)=f_V(q)^{-\omega(\beta_\cL)} \, \text{Tr}^q_V(\beta_\cL;q)
\end{equation*}
where $\text{Tr}^q_V(\beta_\cL;q)$ is the quantum trace of the representation induced by $V$ on a braid diagram $\beta_\cL$, and $f_V(q)$ is a correction factor for the Reidemeister move 1.

In practice, these invariants are computed using the reduced or normalised version. The reduced quantum trace of $\beta_\cL$ is computed as the quantum trace of $\beta_\cL'$, the restriction of $\beta_\cL$ to a subspace of $V^{\otimes n}$ which fixes one of the components. Graphically, this is equivalent to leaving the fixed strand open and closing the rest (Figure \ref{fig:reduced-QT}). When the strands are colored by a finite-dimensional irreducible representation $V$, the resulting map is central, making the invariant independent of the choice of color of the open strand\footnote{Computing the reduced quantum trace will be beneficial in practice, as fixing one of the colors in the braid fixes some of the colors of other arches and results in fewer summations.}. Denote the reduced quantum trace on a link $\mathcal{L}$ with braid diagram $\beta_\cL$ colored by the representation $V$ as
\begin{equation*}
	\widetilde{\text{Tr}}^q_V (\beta_\cL;q):=\text{Tr}^q_V(\beta_\cL';q).
\end{equation*} 

\begin{figure}
    \centering
    \includegraphics[width=0.225\linewidth]{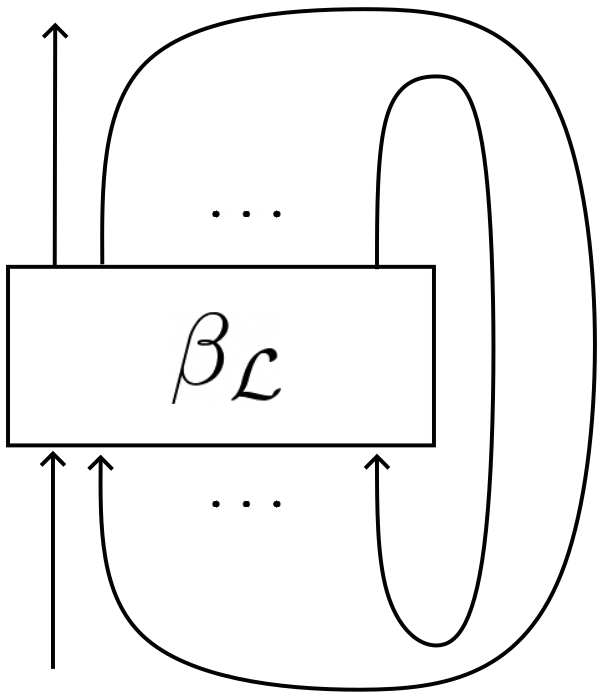}
    \caption{Reduced quantum trace of the braid $\beta_\mathcal{L}$. By convention, the left-most strand is left open, when reading the braid from the bottom up.}
    \label{fig:reduced-QT}
\end{figure}

The action in \eqref{eq:qEi-Fi-Hi} can be used to construct the invariant associated to symmetric representations. For a general irreducible representation, we can use Lemma \ref{lemma:LR-rule-Uqsl3}.

\begin{lema}\label{lemma:RT-inv} Given a braid diagram of a knot $K$, compute the reduced quantum trace over the representation $V_{(n,0)}\otimes V_{(0,m)}$ and evaluate the resulting morphism on the element with weight $(n,m)$. The result is the $\sl_3$ Reshetikhin--Turaev invariant \cite{RT} associated to the finite-dimensional irreducible representation of highest weight $(n,m)$.
\end{lema}

\begin{proof} Let $\beta_K\in B_k$ be a braid whose closure is $K$. Color the strands of $\beta_K$ by the $\Uh(\sl_3)$ representation $V_{(n,0)}\otimes V_{(0,m)}$. Take the reduced quantum trace by leaving the left-most strand open and evaluate the highest weight vector on the open strand. Since the closure of $\beta_K$ is a knot and the maps induced by the braid in the representation category of are morphisms, every strand will be colored by the irreducible component spanned by the eigenvector with eigenvalue $(n,m)$, which corresponds to $V_{(n,m)}$, and the resulting trace is an endomorphism of $V_{(n,m)}$ coinciding with the $\sl_3$ quantum knot invariant associated to $V_{(n,m)}$.
\end{proof}

\section{\texorpdfstring{The large color $R$-matrix}{The large color R-matrix}}\label{sec:construction}

\subsection{Symmetric Verma representations}

It was shown in \cite{gruen22} that there exist $\Uq(\sl_3)$ Verma modules constructing the invariant $F_K^{\sl_3,sym}$ for a variety of knots, including positive braid knots. In analogy with the finite-dimensional case, the goal is to compute the action of $\sl_3$ on the tensor product of such representations. 

The Verma module denoted $V_x^1$ is generated by the basis
\begin{equation*}
    \{v_a:=|a_1,a_2\rangle \,|\, a_1\geq a_2\geq 0 \},
\end{equation*}
and actions given by
\begin{alignat}{3}
    \label{eq:Uhsl3-action-1-1} & K_1\cdot v_a = x^\frac12\,q^\frac{-2a_1+a_2}{2}\,v_a, \quad && E_1\cdot v_a=[a_1-a_2]_q\,v_{a-e^1}, \quad && F_1\cdot v_a = [a_1]_{x,q}\,v_{a+e^1} \\ \label{eq:Uhsl3-action-1-2} & K_2\cdot v_a = q^\frac{a_1-2a_2}{2}\,v_a, \quad && E_2\cdot v_a = [a_2]_q\,v_{a-e^2}, \quad && F_2\cdot v_a = [a_1-a_2]_q\,v_{a+e^2}
\end{alignat}
where $e^1,e^2$ denote the elements $|1,0\rangle$ and $|0,1\rangle$ respectively. The irreducible component of $V_{q^n}^1$ spanned by $v_0$ is isomorphic to $V^q_{(n,0)}$, where we send $v_a=|a_1,a_2\rangle$ to $z_1^{n-a_1}z_2^{a_1-a_2}z_3^{a_3}$ in the $V^q_{(n,0)}$ polynomial basis.

Analogously, one can write the Verma module $V_x^2$, which has the same underlying vector space as $V_x^1$, as the module generated by the basis
\begin{equation*}
    \{v^a:=|a_3,a_4\rangle \,|\, a_3\geq a_4\geq 0 \},
\end{equation*}
and actions given by 
\begin{alignat}{3}
    \label{eq:Uhsl3-action-2-1} & K_1\cdot v^a = q^\frac{a_1-2a_2}{2}\,v^a, \quad && E_1\cdot v^a = [a_2]_q\,v^{a-e^2}, \quad && F_1\cdot v^a = [a_1-a_2]_q\,v^{a+e^2} \\
    \label{eq:Uhsl3-action-2-2} & K_2\cdot v^a = x^\frac12\,q^\frac{-2a_1+a_2}{2}\,v^a, \quad && E_2\cdot v^a=[a_1-a_2]_q\,v^{a-e^1}, \quad && F_2\cdot v^a = [a_1]_{x,q}\,v^{a+e^1}
\end{alignat}
such that $v^0\in V^2_{q^m}$ spans $V^q_{(0,m)}$.

\subsection{Action on the tensor product}

In order to generalize Lemma \ref{lemma:RT-inv} to the infinite-dimensional case, we need to compute the $R$ matrix associated to $V_{x}^{1}\otimes V_{y}^{2}$. In particular, we need to find the action of the elements appearing in Equation \eqref{eq:Bur-formula} on the tensor product. It is convenient to use the alternative choice of generators
\begin{equation*}
    e_i=K_i^{\frac12} E_i, \hspace{2em} f_i=K_i^{-\frac12}F_i \hspace{3em} i=1,2,3.
\end{equation*}

Denote by $w_a:=|a_1,a_2,a_3,a_4\rangle$ an element of the basis of $V_x^1\otimes V_y^2$, where $|a_1,a_2\rangle$ and $|a_3,a_4\rangle$ are basis elements of $V_x^1$ and $V_y^2$ respectively. The actions of the generators $\{e_i,f_i\}_{i=1,2,3}$ are the following.
\begin{alignat}{3}
    \label{eq:e1-action} &e_1\cdot w_a = \left(x^\frac14 \,q^{\frac{a_2-2a_1+2}{4}+\frac{a_3-2a_4}{2}}\,[a_1-a_2]_q \right)\,w_{a-b^1} + \left(q^\frac{a_3-2a_4+2}{4}\,[a_4]_q\right) \, w_{a-b^4}, \\
    &f_1 \cdot w_a = \left(x^{-\frac14}q^{-\frac{a_2-2a_1-2}{4}}[a_1]_{x,q}\right) \, w_{a+b^1} + \left( x^{-\frac12}q^{\frac{-a_2+2a_1}{2}-\frac{a_3-2a_4-2}{4}}[a_3-a_4]_q\right) \,w_{a+b^4}, \\
    &e_2 \cdot w_a = \left(y^\frac12\,q^{\frac{a_1-2a_2+2}{4}}\,q^{\frac{a_4-2a_3}{2}}\,[a_2]_q\right) \,w_{a-b^2}+\left(y^\frac14\,q^{\frac{a_4-2a_3+2}{2}}\,[a_3-a_4]_q\right) \,w_{a-b^3}, \\
    &f_2 \cdot w_a = \left(q^{-\frac{a_1-2a_2-2}{4}}\,[a_1-a_2]_q\right)w_{a+b^2}+\left(y^{-\frac14}\,q^{-\frac{a_2-2a_2}{2}}\,q^{-\frac{a_4-2a_3-2}{4}}\,[a_3]_{y,q}\right)w_{a+b^3}, \\
    &e_3 \cdot w_a = \left(x^{\frac14}\,y^\frac12\,q^{\frac{a_1-3a_2+3}{4}}\,q^{-\frac{a_3+a_4}{2}}\,[a_2]_q\right)w_{a-b^1-b^2}+\left(-y^\frac14\,q^{\frac{-3a_3+a_4+1}{4}}\,[a_4]_q\right)w_{a-b^3-b^4} \\
    &\hspace{4em}+\left(y^\frac12\,(q-q^{-1})\,q^{\frac{a_1-2a_2+2}{4}}\,q^{-\frac{3a_3}{4}}\,[a_2]_q\,[a_4]_q\right)w_{a-b^2-b^4}, \nonumber \\
    \label{eq:f3-action} &f_3\cdot w_a=\left(-x^{-\frac14}\,q^{\frac{3a_2-a_1+1}{4}}\,[a_1]_{x,q}\right)w_{a+b^1+b^2} +\left(x^{-\frac12}\,y^{-\frac14}\,q^{\frac{a_1+a_2}{2}}\,q^{\frac{-a_4+3a_1+3}{4}}\,[a_3]_{y,q}\right)w_{a+b^3+b^4} \\
    &\hspace{4em}+\left(x^{-\frac14}\,y^{-\frac14}\,(q-q^{-1})\,q^{\frac{3a_2}{4}}\,q^{-\frac{a_4-2a_3-2}{4}}\,[a_1]_{x,q}\,[a_3]_{y,q}\right)w_{a+b^1+b^3},\nonumber
\end{alignat}
where $b^1, b^2, b^3$ and $b^4$ are the elements of $\mathbb{N}^4$ with a 1 on the $i^{\text{th}}$ component and 0 on the rest.

We observe that the action of the generators on a basis element lies in a two- or three-dimensional subspace of the tensor product of Verma modules. To compute the action of the $R$-matrix in \eqref{eq:Bur-formula}, we must understand how arbitrary powers of the generators act on this representation. This computation is facilitated by the $q$-Multinomial theorem.

\begin{lema}[$q$-Multinomial theorem, \cite{Foata-Han}]\label{lemma:operators-q-commute}
Let $A, B$ and $C$ be operators satisfying the following $q$-commutation relations
\begin{equation*}
    CA=q AC, \hspace{2em } CB = q BC, \hspace{2em} BA = q AB.
\end{equation*}
Then,
\begin{equation}
    (A+B)^n= \sum_{0\leq k \leq n} \qbinom{n}{k} A^k B^{n-k}
\end{equation}
and
\begin{equation}
    (A+B+C)^n = \sum_{\substack{0\leq k_1,k_2,k_3 \leq n_3 \\ k_1+k_2+k_3=n_3}} \qbinom{n_3}{k_1\;k_2\;k_3} A^{k_1} B^{k_2} C^{k_3} 
\end{equation}
\end{lema}

Split the action of the generators into a sum of two or three different operators, such that the image of each operator is spanned by a unique basis element. For $\{e_i,f_i\}_{i=1,2}$ there will be two of these operators, whereas for $\{e_3,f_3\}$ there will be three; and these operators will either $q$-commute or $q^{-1}$-commute with each other. Therefore, we can apply Lemma \ref{lemma:operators-q-commute} to write an expression for an arbitrary power of the generators as a sum of basis elements with coefficients $q$--hypergeometric functions.

\subsection{\texorpdfstring{The large color $\sl_3$ $R$-matrix}{The large color sl3 R-matrix}}

Let $|\textbf{a},\textbf{b}\rangle$ denote a basis element of $\left(V_x^1\otimes V_y^2\right)\otimes\left(V_x^1\otimes V_y^2\right)$, with $|\textbf{a}\rangle$ and $|\textbf{b}\rangle$ of the form $|\textbf{a}\rangle=|a_1,a_2,a_3,a_4\rangle$.

Let $\textbf{r}$ be the set of internal labels,
\begin{equation*}
    \textbf{r}=\left\{n_1,n_2,n_3,k_{e_{1}},k_{e_{2}},k_{e_{31}},k_{e_{32}},k_{f_{1}},k_{f_{2}},k_{f_{31}},k_{f_{32}}\right\}.
\end{equation*}
Then our $R$-matrix expressions can be simplified by defining the three vectors
\begin{equation*}
    \textbf{n} = (n_1,n_2), \qquad \textbf{k}_e = (k_{e_{1}},k_{e_{2}}), \qquad \textbf{k}_f =(k_{f_{1}},k_{f_{2}})
\end{equation*}
and the overline operation (for $i = 1, 2)$:
\begin{equation*}
    \overline{n_i} := n_i+n_3, \qquad \overline{k_e}_{{}_i} := k_{e_i}+k_{e_{3i}}, \qquad \overline{k_f}_{{}_i} := k_{f_i}+k_{f_{3i}}.
\end{equation*}
We extend this in the obvious way for the vectors above as well as linear combinations. Finally, define the swap operation by $\textbf{m}^t=(m_1,m_2)^t=(m_2,m_1)$.

The R-matrix of $\sl_3$ associated to the representation $V_x^{1}\otimes V_y^{2}$ is given by
\begin{equation}\label{eq:R-matrix-general}
    {}_{\sl_3}R(|\textbf{a},\textbf{b}\rangle) = \sum_{\mathcal{S}(\textbf{r})} R\left(|\textbf{a},\textbf{b}\rangle,\textbf{r}\right)|\textbf{b'},\textbf{a'}\rangle,
\end{equation}
where the summand element $R\left(|\textbf{a},\textbf{b}\rangle,\textbf{r}\right)$ is a polynomial in $x^{-1}$, $y^{-1}$, $q$ and $q^{-1}$
\begin{equation}\label{eq:R-matrix}
    \begin{split}
    R\left(|\textbf{a},\textbf{b}\rangle,\textbf{r}\right) = 
    &(-1)^{k_{e_{32}}-k_{f_{32}}} \,q^{f\left(|\textbf{a},\textbf{b}\rangle,\textbf{r}\,\right) + \frac{1}{3} \left(\log_q (x) \log_q (y)+\log_q ^2(x)+\log_q ^2(y)\right)} 
    \\& x^{-\frac{1}{2}(a_1+b_1+a_4+b_4-\frac{1}{2}(\overline{k_e}_{{}_1}+\overline{k_f}_{1})+\overline{(n-k_f)}_{1})} \, y^{-\frac{1}{2}(a_2+b_2+a_3+b_3-\frac{1}{2}(\overline{(n-k_f)}_{{}_2}+\overline{(n+k_e)}_{{}_2}))}
    \\& \left(q^{-a_2};q\right){}_{\overline{k_e}_{{}_2}} \left(q^{- a_4};q\right){}_{\overline{\left(n-k_{e}\right)}_{{}_1}} \left(q^{ b_1}\,x^{-1};q\right){}_{\overline{k_f}_{{}_1}} \left(q^{b_3}\,y^{-1};q\right){}_{\overline{(n-k_{f})}_{{}_2}}  
    \\& \left(q^{(a_2-\overline{k_e}_{{}_2})-(a_1-k_{e_{31}})};q\right){}_{k_{e_1}} \left(q^{a_4-a_3};q\right){}_{(n-k_e)_{{}_2}} \left(q^{b_2-b_1};q\right){}_{k_{f_2}} 
    \\& \left(q^{(b_4+n_3-k_{f_{31}})-(b_3+\overline{(n-k_{f})}_{{}_2})};q\right){}_{(n-k_f)_{{}_1}}
    \prod_{i=1,2} \frac{\left(q^{n_i};q^{-1}\right){}_{k_{e_i}}\left(q^{n_i};q^{-1}\right){}_{k_{f_i}}}{\left(q^{-1}\right){}_{n_i}\left(q\right){}_{k_{e_i}}\left(q\right){}_{k_{f_i}}} 
    \\&\frac{\left(q^{n_3};q^{-1}\right){}_{k_{e_{32}}}\left(q^{-k_{e_{32}}};q\right){}_{k_{e_{31}}}\left(q^{n_3};q^{-1}\right){}_{k_{f_{31}}}\left(q^{-k_{f_{31}}};q\right){}_{k_{f_{32}}}}{\left(q^{-1}\right){}_{n_3}\left(q^{-1}\right){}_{k_{e_{31}}}\left(q^{-1}\right){}_{k_{f_{32}}}\left(q\right){}_{k_{e_{32}}}\left(q\right){}_{k_{f_{31}}}}
    \end{split}
\end{equation}
with
\begin{align*}
    f\left(|\textbf{a},\textbf{b}\rangle,\textbf{r}\right) = &\textbf{a}\cdot M\cdot \textbf{b} + (-\underline{\textbf{n}}\cdot \underline{\textbf{n}}+2|\underline{\textbf{n}}|+n_1(n_2-n_3)) \\
    &+ \frac{1}{4}\Bigg[(3n_3-n_2)\okef + (-3n_3+n_1)\okes - 4\ketf \okef + 2\kets(1+\kets)+2\okef\okes \\
    &+ (-n_2-n_3)\okff + (-3n_1+n_3)\okfs - 2\kfts(1+\kfts)+4\kftf(n-k_f)_{1}\\
    &+2\okff\okfs +(-\overline{(2n-k_e)}_{2}+4\onf)\,a_1 +(-\overline{(2n+k_e)}_{1}+4\ons)\,a_2\\
    &+(-\overline{(n-k_e)}_{1} +4(n-k_e)_2)\,a_3 +(4\,\overline{(n-k_e)}_{1}+\overline{(n-k_e)}_{2}-4(n-k_e)_2)\,a_4 \\
    &+ (-4\okff +4 \kfs -\okfs) b_1 +(\okff -4\kfs)b_2+(\overline{(3n-k_f)}_{1}-4\ons)b_3\\
    &+(-\overline{(3n-k_f)}_{2}-4(n_1-n_2))b_4\Bigg],
\end{align*}
and
\begin{equation*}
    M = \begin{pmatrix}
        1 & -\frac12 & -\frac12 & 1 \\
        -\frac12 & 1 & 1 & -\frac12 \\
        -\frac12 & 1 & 1 & -\frac12 \\
        1 & -\frac12 & -\frac12 & 1 \\
    \end{pmatrix}
    , \hspace{3em} \underline{\textbf{n}} = (n_1,n_2,n_3).
\end{equation*}

$\mathcal{S}(\textbf{r})$ is the set of constraints over which the sum is performed
\begin{equation}\label{eq:Sr}
\begin{split}
        \mathcal{S}(\textbf{r}) = &\{n_1\geq k_{e_1},k_{f_1}\geq 0\}\,\cup\,\{n_2\geq k_{e_2},k_{f_2}\geq0\}\,\cup\,\{n_3\geq k_{e_{32}}\geq k_{e_{31}}\geq 0\} \\ &\cup\{n_3\geq k_{f_{31}}\geq k_{f_{32}}\geq 0\}
    \end{split}
\end{equation}

The output state $|\textbf{b'},\textbf{a'}\rangle$ has an implicit dependence to both the input state $|\textbf{a},\textbf{b}\rangle$ and $\textbf{r}$ given by
\begin{alignat}{4}\label{eq:output-input-1}
    &|\textbf{a'}\rangle &&= |\textbf{a}\rangle - |\textbf{k}_e\textbf{'}\rangle, \hspace{1em} \text{where} \hspace{1em} &&|\textbf{k}_e\textbf{'}\rangle &&= |\overline{\textbf{k}_e},\overline{(\textbf{n}-\textbf{k}_e)^t}\rangle \\ \label{eq:output-input-2}
    &|\textbf{b'}\rangle &&= |\textbf{b}\rangle + |\textbf{k}_f\textbf{'}\rangle, \hspace{2em} &&|\textbf{k}_f\textbf{'}\rangle &&= |\overline{\textbf{k}_f},\overline{(\textbf{n}-\textbf{k}_f)^t}\rangle
\end{alignat}
where $|\textbf{a'}\rangle$ is the state assigned to the overstrand and $|\textbf{b'}\rangle$ to the understrand. It follows immediately that $a_i-a_i'\geq 0$, $b_i'-b_i\geq 0$ for all $i=1,\dots,4$ and the internal labels fix two conservation relations between the incoming and outgoing states
\begin{equation}\label{eq:conservation-rels}
    \sum_{i=1,4} (b_i'-b_i)-(a_i-a_i')=0 \qquad \text{and} \qquad \sum_{i=2,3} (b_i'-b_i)-(a_i-a_i')=0.
\end{equation}

\subsection{The symmetric limit}\label{sec:symmetric-limit}

The symmetric invariant $F_K^{\sl_3,sym}$ associated to symmetric representations \cite{EGGKPS, gruen22, Park20} can be recovered as a specialization of $F_K^{\sl_3}$.

\begin{proposicion} Under the assumptions of Theorem \ref{thm:FK-SU(3)},
\begin{equation}\label{eq:sym-limit}
         F_K^{\sl_3}(x,y=1,q)=F_K^{\sl_3,sym}(x,q).
\end{equation}
\end{proposicion}

\begin{proof} When $y=1$, the highest weight vector of the Verma module $V^y_2$ generates the trivial representation $\C$. Since $F_K^{\sl_3}$ is defined using the $R$-matrix of $V_x^1\otimes V_y^2$, the substitution $y=1$ sends ${}_{\sl_3}R$ to the $R$-matrix of the Verma module $V_x^1$, which constructs $F_K^{\sl_3,sym}$ in \cite{gruen22}. Given that the open strand is colored by the element generating $V_x^1\otimes \C \cong V_x^1$ and the closure of the braid is a knot, the result in \eqref{eq:sym-limit} immediately follows.
\end{proof}

As expected, the computation of the general $\sl_3$ invariant captures more information and is more involved than its symmetric counterpart. This becomes apparent when examining the $R$-matrix expression, as the number of internal labels for each crossing grows from three to eleven. Braiding a strand colored with a tensor representation is equivalent to braiding two parallel strands, each colored by a symmetric representation, which justifies the growth of the number of labels that play a role in our construction. Explicitly, if we denote by $R_{z,w}^{i,j}$ the $R$-matrix acting on $V_z^i\otimes V_w^j$, the equality
\begin{equation}
    {}_{\sl_3}R = (\id_{V_x^1}\otimes R_{x,y}^{1,2} \otimes \id_{V_y^2})\circ(R_{x,x}^{1,1} \otimes R_{y,y}^{2,2})\circ(\id_{V_x^1}\otimes R_{y,x}^{2,1} \otimes \id_{V_y^2})
\end{equation}
holds when acting on $\left(V_{x}^1\otimes V_y^2\right)\otimes\left(V_{x}^1\otimes V_y^2\right)$. See Figure \ref{fig:crossing-tensor}.

\begin{figure}
    \centering
    \includegraphics[width=0.5\linewidth]{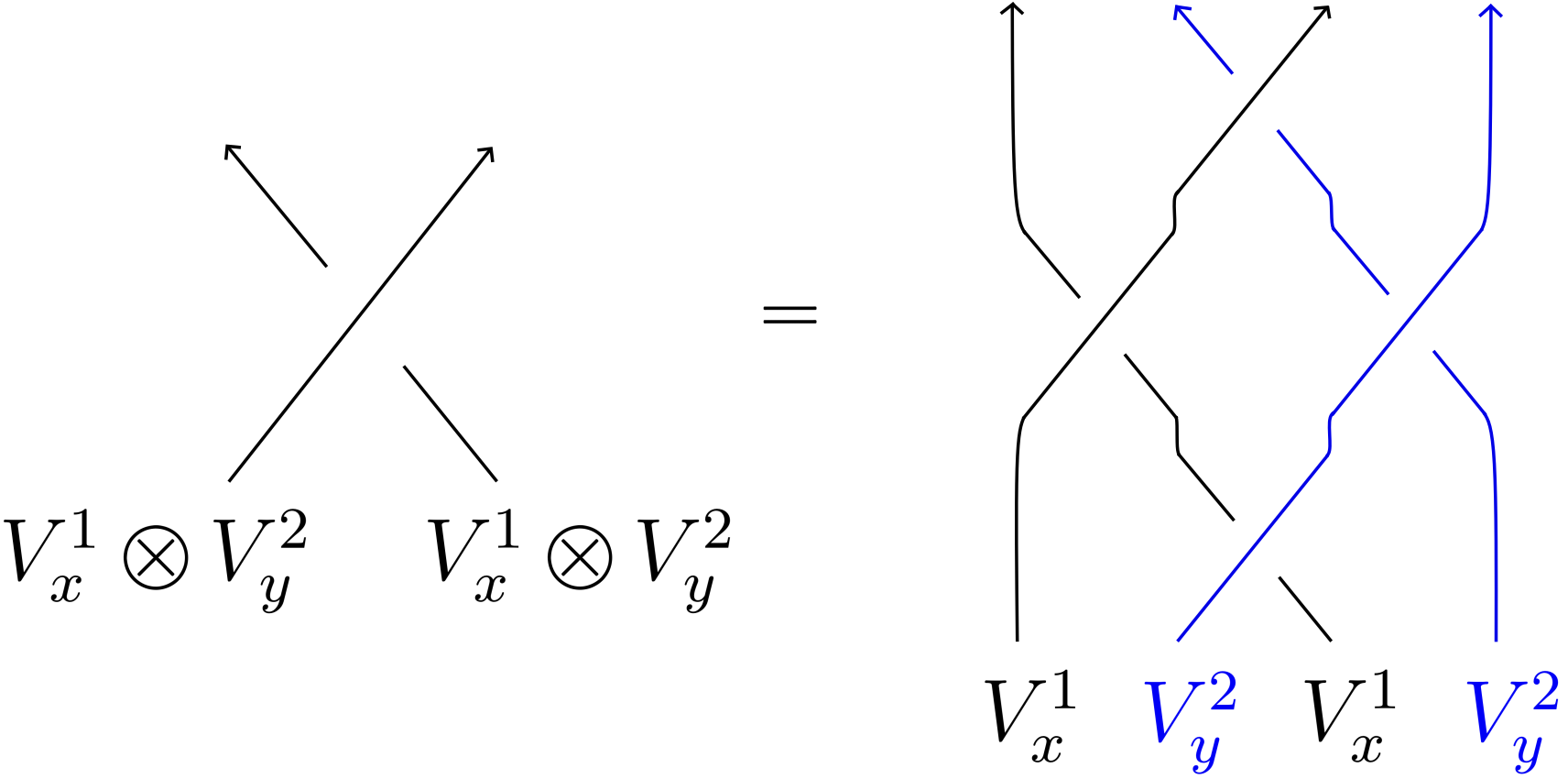}
    \caption{A strand colored with $V_x^1\otimes V_y^2$ can be replaced by two parallel strands, the first colored by $V_x^1$ and the second by $V_y^2$.}
    \label{fig:crossing-tensor}
\end{figure}

\subsection{Proof of Theorem \ref{thm:FK-SU(3)}}

Let us first assume that the reduced quantum trace converges to a power series in $x^{-1}$ and $y^{-1}$, with coefficients Laurent polynomials in $q$. In the specialization $x=q^n$ and $y=q^m$, the basis element $|\textbf{0}\rangle$ spans the finite representation $V_{(n,0)}\otimes V_{(0,m)}$ and recovers the associated $R$-matrix. By Lemma \ref{lemma:RT-inv}, the reduced quantum trace gives the $\sl_3$ Reshetikhin--Turaev invariant associated to $V_{(n,m)}$ of $K$, which proves \eqref{eq:thm-MMR} and \eqref{eq:thm-alexander}.

Now let us show that there is indeed convergence in the quantum trace.

\begin{lema}\label{lemma:proof-1} Denote by $|\textbf{a},\textbf{b}\rangle$ and $|\textbf{b'},\textbf{a'}\,\rangle$ the incoming and outgoing states respectively. Then, the $R$-matrix expression \eqref{eq:R-matrix} is an element of $$x^{-\frac{1}{2}(a_1'+a_4'+b_1+b_4)}y^{-\frac{1}{2}\left(\frac{3}{2}a'_2+a_3'+\frac{1}{2}b_2+b_3+\frac{1}{2}\left(b_2'-a_2\right)\right)}\mathbb{Z}[q,q^{-1},x^{-1},y^{-1}].$$
\end{lema}

\begin{proof}

The $x$ contribution in the R-matrix expression is
\begin{equation*}
    x^{-\frac{1}{2}(a_1+a_4+b_1+b_4-\frac{1}{2}(\overline{k_e}_{{}_1}+\overline{k_f}_{1})+\overline{(n-k_f)}_{1})} \, \left(\frac{q^{ b_1}}{x};q\right){}_{\overline{k_f}_{{}_1}}.
\end{equation*}
Since $\overline{k_f}_{{}_1} \geq 0$, we obtain the following bounds for the $x$-exponent
\begin{alignat*}{2}
    &\text{Highest:} \hspace{3em} &&x^{-\frac{1}{2}(a_1+a_4+b_1+b_4-\frac{1}{2}(\overline{k_e}_{{}_1}+\overline{k_f}_{1})+\overline{(n-k_f)}_{1})} \\
    &\text{Lowest:} \hspace{3em} &&x^{-\frac{1}{2}(a_1+a_4+b_1+b_4-\frac{1}{2}(\overline{k_e}_{{}_1}+\overline{k_f}_{1})+\overline{(n+k_f)}_{1})}
\end{alignat*}
Similarly, the contribution in $y$
\begin{equation*}
    y^{-\frac{1}{2}(a_2+a_3+b_2+b_3-\frac{1}{2}(\overline{(n-k_f)}_{{}_2}+\overline{(n+k_e)}_{{}_2}))} \, \left(\frac{q^{b_3}}{y};q\right){}_{\overline{(n-k_{f})}_{{}_2}}
\end{equation*}
is bounded as a polynomial in $y$ by the following exponents
\begin{alignat*}{2}
    &\text{Highest:} \hspace{3em} &&y^{-\frac{1}{2}(a_2+a_3+b_2+b_3-\frac12({\overline{(n-k_f)}_{{}_2}+\overline{(n+k_e)}_{{}_2}))}} \\
    &\text{Lowest:} \hspace{3em} &&y^{-\frac{1}{2}(a_2+a_3+b_2+b_3-\frac12(-3{\overline{(n-k_f)}_{{}_2}+\overline{(n+k_e)}_{{}_2}))}}
\end{alignat*}

The desired result follows after applying the conservation relations in Eq. \eqref{eq:conservation-rels}.
\end{proof}

Let $s$ be the number of strands of $\beta_K$. Let $\textbf{i}_k=|i_{k,1},i_{k,2},i_{k,3},i_{k,4}\rangle$ be a basis element of $V_x^{1}\otimes V_y^{2}$, labelled by $k=1,\dots,s$. Define the tensor element $\langle \textbf{0},\textbf{i}_2,\dots,\textbf{i}_s|\beta_K|\textbf{0},\textbf{i}_2,\dots,\textbf{i}_s\rangle$ by reading each crossing in the braid $\beta_K$ as an $R$-matrix, evaluating the corresponding map at $|\textbf{0},\textbf{i}_2,\dots,\textbf{i}_s\rangle$, where $s$ is the number of strands of $\beta_K$, and extracting the coefficient of the state $|\textbf{0},\textbf{i}_2,\dots,\textbf{i}_s\rangle$. Here, $|\textbf{0},\textbf{i}_2,\dots,\textbf{i}_s\rangle$ denotes the tensor product of the states $\textbf{0},\dots,\textbf{i}_s$.

\begin{lema}\label{lemma:proof-2} The tensor element $\langle \textbf{0},\textbf{i}_2,\dots,\textbf{i}_s|\beta_K|\textbf{0},\textbf{i}_2,\dots,\textbf{i}_s\rangle$ is in $\mathbb{Z}[q,q^{-1},x^{-1},y^{-1}]$.
\end{lema}

\begin{proof}
As a consequence of Lemma \ref{lemma:proof-1}, the $x$ variable has the same behaviour as in the symmetric case, for which we refer the reader to \cite{gruen22}. On the other hand, the $y$ variable requires a more subtle analysis. Notice that the $y$-exponent is now a multiple of $\frac14$ instead of $\frac12$ and has a positive contribution coming from each crossing. Up to an overall $\frac{1}{2}$ factor, the contribution to the $y$-exponent given by a positive crossing going from the state $|\textbf{a},\textbf{b}\rangle$ to $|\textbf{b}',\textbf{a}'\rangle$ is 
\begin{equation}\label{eq:y-exponent}
    -\frac{1}{2}\left(3\,a'_2+b_2'+b_2-a_2\right)-a_3'-b_3
\end{equation}

The fact that there are no positive nor fractional powers contributions in the overall $y$-exponent (up to a $\frac12$ factor) follows from studying the movement of the second component of the weights inside the braid, as well as the identification of the bottom and top strands.

First, we notice that the second component of the labels of all four states $\textbf{a},\textbf{b},\textbf{a}'$ and $\textbf{b}'$ make an appearance in \eqref{eq:y-exponent}. Since each weight in the braid is associated to exactly two crossings, after identifying the ends of the braid, the two contributions will add up to guarantee the $y$-exponent to be in $\frac12\Z$.

This same pairing also guarantees that the overall exponent is actually in $-\frac12\N$. For each crossing $\sigma_k$, the positive contribution to the $y$-exponent $\frac{1}{4}a_2$ is weighted out by the negative contribution $-\frac{3}{4}a_2$ coming from the previous $\sigma_{k-1}$. In the case when the $k^\text{th}$ strand connects to the bottom of the braid, the contribution will get cancelled by the first crossing which connects to the $k^\text{th}$ strand from the top, which introduces either $-\frac{1}{4}a_2$ (if it corresponds to the understrand) or $-\frac34 a_2$ (if it corresponds to the overstrand).

Once this pairing has been done, the overall $y$-exponent will have contributions parallel to those of the overall $x$-exponent, except that now the labels' contributions come from the second and third components instead of the first and fourth. Applying the conservation relations in $\eqref{eq:conservation-rels}$, together with the identification of the bottom and top strands, guarantees that there will be an even number of contributions for each label component. As a result, the final series will have integral exponents in $x^{-1}$ and $y^{-1}$.
\end{proof}

\begin{lema}\label{lemma:proof-3} The product $$x^{\frac12\left(i_{2,1}+i_{2,4}+\dots+i_{s,1}+i_{s,4}\right)}\,y^{\frac12\left(i_{2,2}+i_{2,3}+\dots+i_{s,2}+i_{s,3}\right)}\langle \textbf{0},\textbf{i}_2,\dots,\textbf{i}_s|\beta_K|\textbf{0},\textbf{i}_2,\dots,\textbf{i}_s\rangle$$ is an element of $\mathbb{Z}[q,q^{-1},x^{-1},y^{-1}]$.
\end{lema}

\begin{proof} 
We will prove the statement by showing independently that the $x^{-1}$--exponent and $y^{-1}$-exponent in the tensor element $\langle \textbf{0},\textbf{i}_2,\dots,\textbf{i}_s|\beta_K|\textbf{0},\textbf{i}_2,\dots,\textbf{i}_s\rangle$ are bounded from below by $\frac12\left(i_{2,1}+i_{2,4}+\dots+i_{s,1}+i_{s,4}\right)$ and $\frac12\left(i_{2,2}+i_{2,3}+\dots+i_{s,2}+i_{s,3}\right)$ respectively.

As in Lemma \ref{lemma:proof-2}, the $x$ variable can be bounded by the same argument as in the symmetric case. Thus, we will focus only in the $y$ variable, by showing that each $\textbf{i}_k$ adds a contribution to the $y$-exponent which is bounded by $-\frac12\left(i_{k,2}+i_{k,3}\right)$. The proof uses that every $\sigma_{k-1}$ appears in the braid, which is guaranteed by the condition that the closure of the braid is a knot.

Assume that there are no $\sigma_{k}$ before the first $\sigma_{k-1}$ appearance, such that the bottom right strand of that first $\sigma_{k-1}$ is $\textbf{i}_k$ as in Figure \ref{fig:lemma-case-1}. Recall that the positive contribution of the bottom left strand of the $\sigma_{k-1}$ will be cancelled out either by a $\sigma_{k-2}$ crossing to its left or, in the case there are none, by one of the top crossings. The overall $y^{-1}$-exponent coming from that $\sigma_{k-1}$ will be greater than $\frac14 i_{k,2}+\frac12i_{k,3}$. The top crossing will contribute either $\frac34 i_{k,2}$ or $\frac14 i_{k,2}$, giving the lower bound $\frac12\left(i_{k,2}+i_{k,3}\right)$ as desired.

\begin{figure}
\centering
\begin{minipage}{.525\textwidth}
  \centering
\begin{subfigure}{.45\textwidth}
  \centering
  \includegraphics[height=7cm,keepaspectratio]{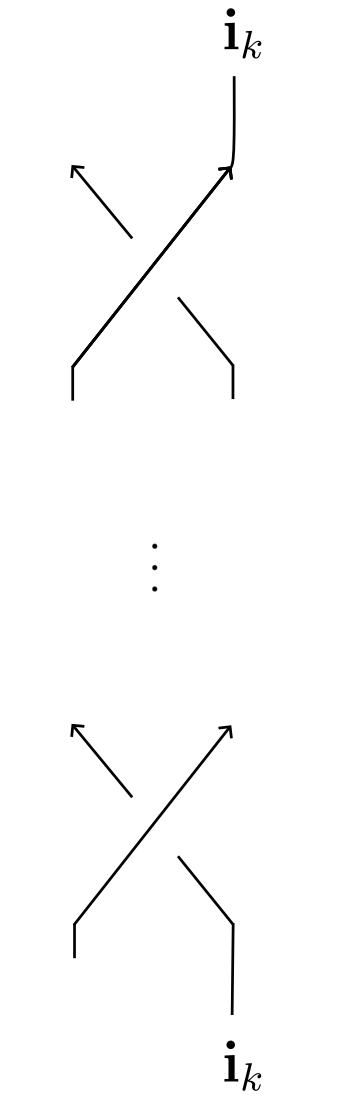}
  \label{fig:sub1}
\end{subfigure}
\begin{subfigure}{.45\textwidth}
  \centering
  \includegraphics[height=7cm,keepaspectratio]{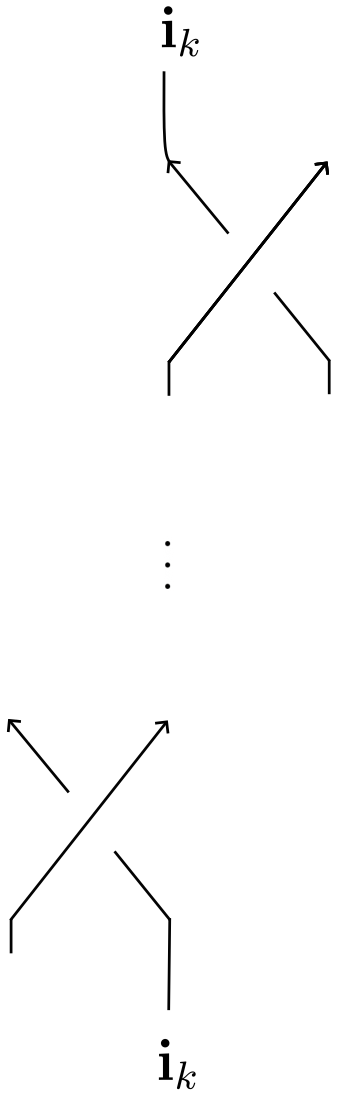}
  \label{fig:sub2}
\end{subfigure}
\caption{The two possibilities when there are no appearances of $\sigma_k$ before the first $\sigma_{k-1}$.}
\label{fig:lemma-case-1}
\end{minipage}
\hfill
\begin{minipage}{.425\textwidth}
  \centering
  \includegraphics[height=7cm,keepaspectratio]{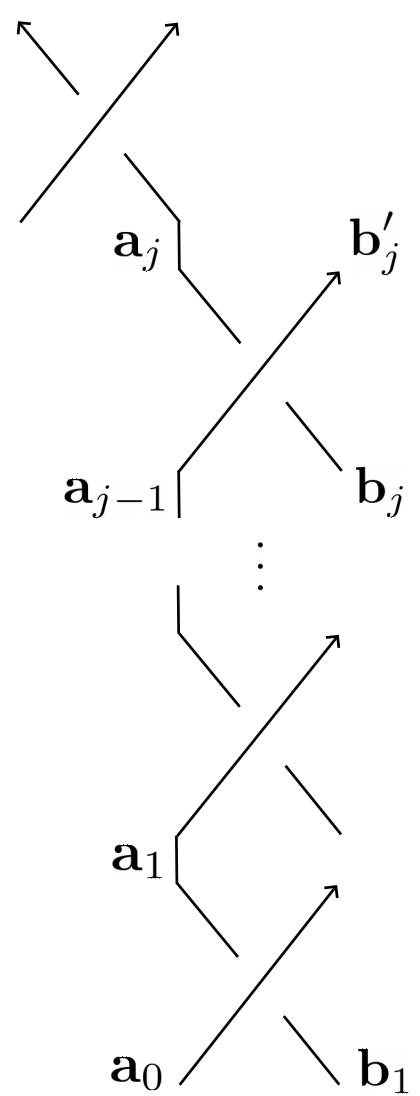}
  \captionof{figure}{$j$ appearances of $\sigma_k$ before the first $\sigma_{k-1}$.}
  \label{fig:lemma-case-2}
\end{minipage}
\end{figure}

For the second scenario, assume there are $j$ appearances of $\sigma_{k}$ before the first $\sigma_{k-1}$ as in Figure \ref{fig:lemma-case-2}. Let the $i^\text{th}$ $\sigma_{k}$ go from $|\textbf{a}_{i-1},\textbf{b}_i\rangle$ to $|\textbf{a}_{i},\textbf{b}_i'\rangle$, where $\textbf{a}_0=\textbf{i}_k$. Its contribution to the $y^{-1}$-exponent is $\frac12\left(\frac12\left(3\, b_{i,2}'+b_{i,2}\right)+\left(b_{i,3}'+b_{i,3} \right)+\frac12\left(a_{i,2}-a_{i-1,2}\right)\right)$. Considering the $j$ crossings altogether, the positive contributions cancel kaleidoscopically and the $y^{-1}$-exponent is bounded from below by 
\begin{equation}\label{eq:lemma-proof-3-case-2}
    \frac12\left(\sum_{i=1}^{j} b_{i,2}'+ b_{i,3}'\right) + \frac14 a_{j,2},
\end{equation}
where we have omitted $-\frac14 a_{0,2}$ as it will get canceled out by the contributions from the top crossings. 

Adding the bottom right strand contribution of the $\sigma_{k-1}$, $\frac14 a_{j,2}+\frac12 a_{j,3}$, we can use $j$ times the conservation law in \eqref{eq:conservation-rels} to bound \eqref{eq:lemma-proof-3-case-2} uniquely by $a_0$, leading to an overall $y^{-1}$ lower bound of $\frac12 \left(a_{0,2}+a_{0,3}\right)=\frac12 \left(i_{k,2}+i_{k,3}\right)$ as desired.
\end{proof}

The reduced quantum trace $\widetilde{\text{Tr}}^q_{V_x^1\otimes V_y^2} (\beta_K)$ is computed as
\begin{equation*}
    \widetilde{\text{Tr}}^q_{V_x^1\otimes V_y^2} (\beta_K) = \sum_{|\textbf{0},\textit{i}_2,\dots,\textbf{i}_s\rangle} q^{-\sum_{k=2}^s |\textbf{i}_k|} \, \langle \textbf{0},\textbf{i}_2,\dots,\textbf{i}_s|\beta_K|\textbf{0},\textbf{i}_2,\dots,\textbf{i}_s\rangle
\end{equation*}
As a direct conclusion of Lemma \ref{lemma:proof-3}, only a finite number of input states $|\textbf{0},\textbf{i}_2,\dots,\textbf{i}_s\rangle$ contribute to each exponent of $x^{-1}$ and $y^{-1}$ in the reduced quantum trace. Moreover, for each input state, the summations over internal labels in $\langle \textbf{0},\textbf{i}_2,\dots,\textbf{i}_s|\beta_K|\textbf{0},\textbf{i}_2,\dots,\textbf{i}_s\rangle$ are also finite, due to the conservation relations \eqref{eq:conservation-rels} and the expression of internal labels in terms of the incoming and outgoing states (\ref{eq:Sr}, \ref{eq:output-input-1}, \ref{eq:output-input-2}). This altogether shows the convergence of the reduced quantum trace.

\section{Examples}\label{sec:examples}

In order to compute $F_K^{\sl_3}$ for some examples, it is more suitable to express the quantum trace as a sum of matrix elements. For notational ease, denote by $R_{\textbf{a},\textbf{b}}^{\textbf{a'},\textbf{b'}}$ the coefficient $\sum_{\mathcal{S}(\textbf{r})}R(|\textbf{a},\textbf{b}\rangle,\textbf{r})$ accompanying $|\textbf{a'},\textbf{b'}\rangle$ in \eqref{eq:R-matrix-general}.

\begin{figure}[h!]
    \centering
    \includegraphics[width=0.35\linewidth]{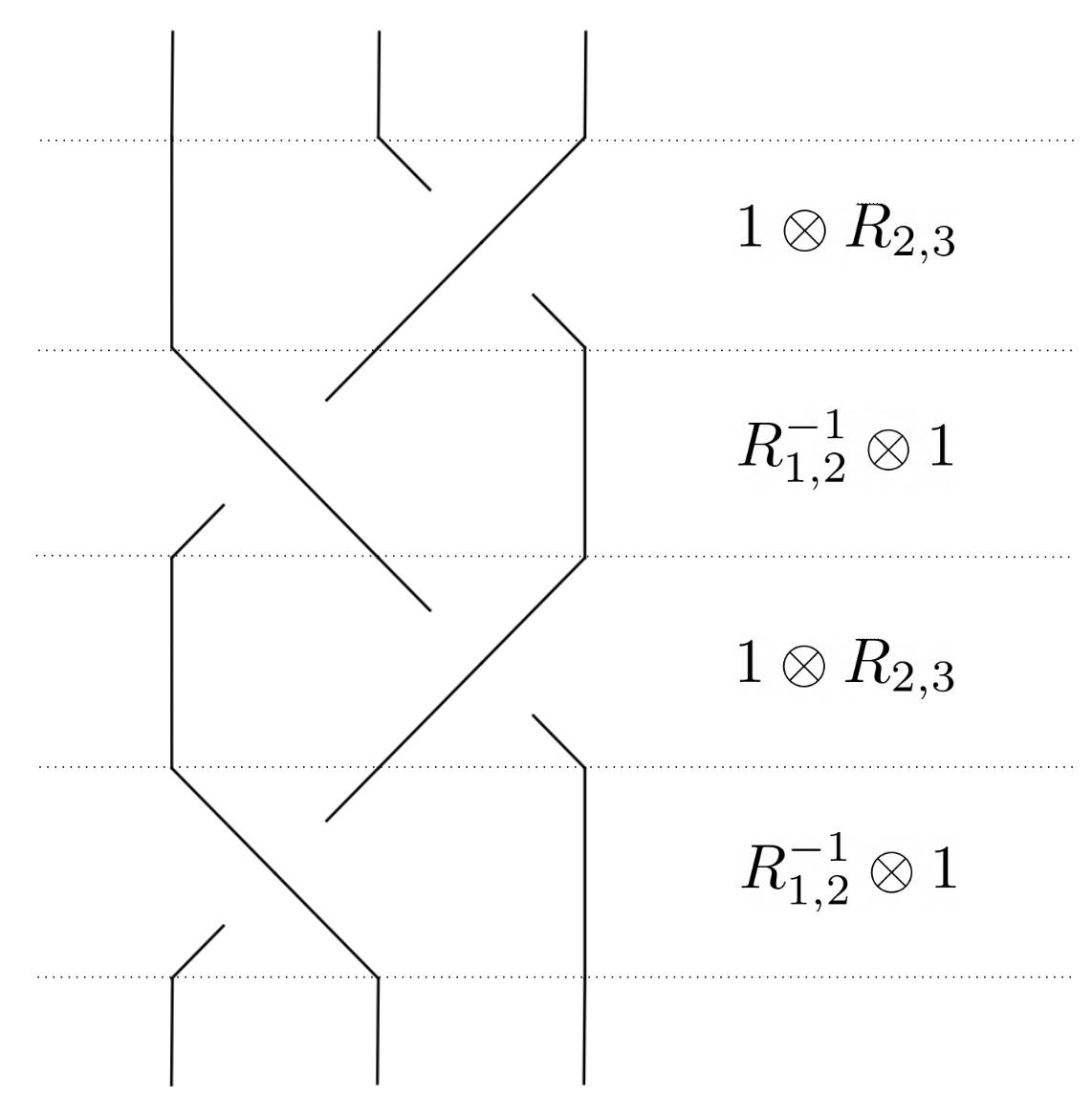}
    \caption{The braid $\sigma_1^{-1}\sigma_2\sigma_1^{-1}\sigma_2$ is divided into 4 segments, each including one of the crossings. To the right, we write the maps associated to each level, which are being composed in the quantum trace.}
    \label{fig:levels}
\end{figure}

Given a braid diagram, divide your braid into as many segments as crossings, such that each segment (which we call \textit{level}) features only one of the crossings. Assign the state $|\textbf{0}\rangle$ to the left incoming strand (when read from the bottom up) and label each crossing's incoming and outgoing strand of the braid diagram. The strands are labeled so that the total sum of labels is constant at each component and level, and the bounds in \eqref{eq:output-input-1}--\eqref{eq:output-input-2} are satisfied. The tensor element $\langle \textbf{0},\textbf{i}_2,\dots,\textbf{i}_s|\beta_K|\textbf{0},\textbf{i}_2,\dots,\textbf{i}_s\rangle$ is the product of the $R$-matrix elements for each incoming and outcoming state at each level.

\subsection{Trefoil}

Let us exemplify the procedure with the simplest knot, the trefoil. Using the braid diagram $\beta_{3_1^{r}}=\sigma_1^3$, the element $\langle \textbf{0},\textbf{b}|\beta_{3_1^{r}}|\textbf{0},\textbf{b}\rangle$ is
$$R_{\textbf{0},\textbf{b}}^{\textbf{b},\textbf{0}}R_{\textbf{b},\textbf{0}}^{\textbf{b},\textbf{0}}R_{\textbf{b},\textbf{0}}^{\textbf{0},\textbf{b}}.$$
The first and third $R$-matrix elements are just monomials in $x$ and $y$,
$$R_{\textbf{0},\textbf{b}}^{\textbf{b},\textbf{0}} = R_{\textbf{b},\textbf{0}}^{\textbf{0},\textbf{b}} = x^{\frac{1}{2}(-b_1-b_4)}y^{\frac{1}{2}(-b_2-b_3)},$$
while the middle $R$-matrix involves a sum over five internal labels.

The reduced quantum trace is
\begin{equation}
    F_{3_1^{r}}^{\sl_3}(x,y,q) = \sum_{\textbf{b}} q^{-|\textbf{b}|} R_{\textbf{0},\textbf{b}}^{\textbf{b},\textbf{0}}R_{\textbf{b},\textbf{0}}^{\textbf{b},\textbf{0}}R_{\textbf{b},\textbf{0}}^{\textbf{0},\textbf{b}}
\end{equation}
where the sum goes over $\textbf{b}\in\left\{|b_1,b_2,b_3,b_4\rangle | b_1\geq b_2\geq 0, b_3\geq b_4\geq 0\right\}$. Once we plug in the correct values for the matrix elements, the result is an element of $\mathbb{Z}[q^{\pm 1}][[x^{-1},y^{-1}]]$:
\begin{equation}\label{eq:trefoil-1}
\begin{split}
    F_{3_1^{r}}^{\sl_3}(x,y,q)&=\frac{1}{x^2y^2}\left[\left(1+\frac{1}{qy}+\frac{1-q}{(qy)^2}+\frac{-q^2-q+1 }{(qy)^3}\right)\right.
    \\&+\left(1+\frac{2}{qy}+\frac{2-q}{(qy)^2}+\frac{-q^2-2 q+2}{(qy)^3}\right)\frac{1}{qx}
    \\&+\left((1-q)+\frac{2-q}{qy}+\frac{3-2 q}{(qy)^2}-\frac{3 (q-1)}{(qy)^3}\right)\frac{1}{(qx)^2} 
    \\&+ \left((-q^2-q+1)+\frac{-q^2-2 q+2}{qy}-\frac{3 (q-1)}{(qy)^2}-\frac{4 (q-1)}{(qy)^3}\right)\frac{1}{(qx)^3} 
    \\&+\left.\dots\right]
\end{split}
\end{equation}

\subsection{Cinquefoil}

Applying the same procedure to the braid diagram $\beta_{5_1}=\sigma_1^5$ gives the $\sl_3$-invariant for the $5_1$ knot.
\begin{equation}\label{eq:cinquefoil}
\begin{split}
    F_{5_1}^{\sl_3}(x,y,q)&= \frac{1}{x^4y^4}\left[\left(1+\frac{1}{qy}+\frac{1-q}{(qy)^2}-\frac{1-q}{(qy)^3}\right)\right.
    \\&+\left(1+\frac{2}{qy}+\frac{2-q}{(qy)^2}+\frac{2(1-q)}{(qy)^3}\right)\frac{1}{qx}
    \\&+\left((1-q)+\frac{2-q}{qy}+\frac{3-2 q}{(qy)^2}+\frac{3(1-q)}{(qy)^3}\right)\frac{1}{(qx)^2}
    \\&+\left.\left((1-q)+\frac{2(1-q)}{qy}+\frac{3(1-q)}{(qy)^2}+\frac{4(1-q)}{(qy)^3}\right)\frac{1}{(qx)^3}+\dots\right]
\end{split}
\end{equation}

\subsection{Comparison with previous work}

In \cite{Park20}, Park computed the series $F_K^{\sl_3}$ for any torus knot $K$. In order to compare results, we need to change from our convention to theirs.

First, we need to shift our $x$ and $y$ variables via $x\mapsto q^{-1}\,x$, $y\mapsto q^{-1}\,y$.

Notice that the invariant in \cite{Park20} is written in balanced and unreduced expression, whereas above we write the invariant in negative reduced expression.  Shifting from reduced to unreduced corresponds to multiplying by the normalization factor
\begin{equation*}
   (x^{\frac12}- x^{-\frac12})(y^\frac12 - y^{-\frac12})((xy)^{-\frac12} - (xy)^{\frac12}).
\end{equation*}

To go from the negative expression to the balanced one, we need to impose the $SL(3,\C)$ Weyl symmetry, given by the action of the symmetric group of degree three $S_3$ on the vector of variables $\textbf{x}=(x,y,(xy)^{-1})$.

If we denote 
\begin{equation}
    F_K^{\sl_3,neg}(x,y,q):=-(x^{\frac12}- x^{-\frac12})(y^\frac12 - y^{-\frac12})((xy)^{-\frac12} - (xy)^{\frac12}) \, F_K^{\sl_3}(q\,x,q\,y,q)
\end{equation}
the negative expansion of the unreduced invariant, and $F_K^{\sl_3,pos}(x,y,q):=F_K^{\sl_3,neg}(x^{-1},y^{-1},q)$ the associated positive expansion, the balanced $F_K^{\sl_3}$ is given by
\begin{equation}
    F_K^{\sl_3,balanced}(x,y,q)=\frac{1}{2}\left(-\sum_{s\in S_3} F_K^{\sl_3,neg}(\pi(s(\textbf{x})) + \sum_{s\in S_3} F_K^{\sl_3,pos}(\pi(s(\textbf{x}))\right)
\end{equation}
where $\pi(\textbf{x})$ is the projection on the first two components of $\textbf{x}$. The normalization by a factor of $\frac12$ comes from our original series already showing symmetry in $x$ and $y$.

\textbf{Examples}

Reducing $F_{3_1^r}^{\sl_3}$ to the symmetric variant $F_{3_1^r}^{\sl_3,sym}$ studied in \cite{EGGKPS,gruen22}, simply involves setting $y = 1$ and using the Weyl symmetry $x^{-1}\rightarrow q^3 x$. This gives, up to overall $x,y$ and $q$ power:
\begin{equation*}\begin{split}
    F_{3_1^r}^{\sl_3}(q^{-3}x^{-1},1,q) =1 + (1 + q) (qx) + (1 + q - q^3) (q x)^2 + 
 (1 + q - q^3 - 2 q^4 - q^5) (qx)^3 + \dots   
\end{split}
\end{equation*}

Similarly, $F_{5_1^r}^{\sl_3}$ reduces to $F_{5_1^r}^{\sl_3,sym}$:
\begin{equation*}
    F_{5_1^r}^{\sl_3}(q^{-3}x^{-1},1,q) = 1 + (1+q)(q x) + (1+q-q^3)(qx)^2 + (1+q-q^3-q^4)(q x)^3+\dots
\end{equation*}
which matches the result in Table 2 of \cite{gruen22}, up to overall $x,y$ and $q$ power.

In order to compare with Park, we first expand the invariant as a series in $q^{-1}$. In the case of the right-hand trefoil, Equation \eqref{eq:trefoil-1}  yields
\begin{equation*}
    \begin{aligned}
        F_{3_1^r}^{\sl_3}(x^{-1},y^{-1},q) = \,&x^2 y^2 + \frac{1}{q} \left(-x^5 y^2-x^4 y^2+x^3 y^2-x^2 y^5-x^2 y^4+x^2 y^3\right) \\
        &+ \frac{1}{q^2} \left(-x^6 y^2-x^5 y^3-x^5 y^2-x^4 y^3+x^4 y^2-x^3 y^5-x^3 y^4+2 x^3 y^3\right.\\
        &\quad\left.-x^2 y^6-x^2 y^5+x^2 y^4\right) \\
        &+ \frac{1}{q^3} \left(x^6 y^4-2 x^6 y^3-x^6 y^2-2 x^5 y^3+x^5 y^2+x^4 y^6\right.\\
        &\quad\left.-2 x^4 y^4+2 x^4 y^3-2 x^3 y^6-2 x^3 y^5+2 x^3 y^4-x^2 y^6+x^2 y^5\right) \\
        &+ \frac{1}{q^4} \left(-x^6 y^6+x^6 y^5-x^6 y^4-2 x^6 y^3+x^6 y^2+x^5 y^6\right.\\
        &\quad\left.-3 x^5 y^4+2 x^5 y^3-x^4 y^6-3 x^4 y^5+3 x^4 y^4-2 x^3 y^6\right.\\
        &\quad\left.+2 x^3 y^5+x^2 y^6\right) \\
        &+ \frac{1}{q^5} \left(2 x^6 y^6-4 x^6 y^4+2 x^6 y^3-4 x^5 y^5+3 x^5 y^4-4 x^4 y^6+3 x^4 y^5\right.\\
        &\quad\left.+2 x^3 y^6\right) + \frac{1}{q^6} \left(-5 x^6 y^5+3 x^6 y^4-5 x^5 y^6+4 x^5 y^5+3 x^4 y^6\right) \\
        &+ \frac{1}{q^7} \left(-6 x^6 y^6+4 x^6 y^5+4 x^5 y^6\right) + \frac{1}{q^8} \left(5 x^6 y^6\right) + \dots
    \end{aligned}
\end{equation*}

After renormalising the variables $x$ and $y$, multiplying by the normalization factor and dropping higher order exponents, the negative expansion turns into
\begin{equation*}
\begin{split}
    F_{3_1^r}^{\sl_3,pos}(x,y,q)= &q^4\, x y - q^5\, (x^3 y +x y^3) + q^6\, (-x^4 y + x^4 y^3 -x y^4 + x^3 y^4) + 
 q^7 \,(x^5 y^3 + x^3 y^5) \\&- q^8\,x^5 y^5 + \dots
\end{split}
\end{equation*}
Note that the symmetry interchanging $x$ and $y$ is already present. Forcing the $S_3$ symmetry, the positive expansion of $F_K$ contributes
\begin{equation*}
    \begin{split}
        &q^4 \left(x y+\frac{1}{x}+\frac{1}{y}\right) + q^5 \left(-\frac{1}{x^3 y^2}-x^3 y-\frac{1}{x^2 y^3}-\frac{x^2}{y}-x y^3-\frac{y^2}{x}\right) \\
        &+ q^6 \left(x^4 y^3-\frac{1}{x^4 y^3}-x^4 y+\frac{1}{x^4 y}+x^3 y^4-\frac{1}{x^3 y^4}-\frac{x^3}{y}+\frac{y}{x^3}-x y^4+\frac{1}{x y^4}+\frac{x}{y^3}-\frac{y^3}{x}\right) \\
        &+ q^7 \left(x^5 y^3+\frac{1}{x^5 y^2}+x^3 y^5+\frac{y^2}{x^3}+\frac{1}{x^2 y^5}+\frac{x^2}{y^3}\right) + q^8 \left(-x^5 y^5-\frac{1}{x^5}-\frac{1}{y^5}\right) + \dots
    \end{split}
\end{equation*}

Finally, balancing with the positive expansion, one gets the series
\begin{equation*}
    \begin{split}
        &F_{3_1^r}^{\sl_3,balanced}(x,y,q) = q^4 \left(x y-\frac{1}{x y}-x+\frac{1}{x}-y+\frac{1}{y}\right) \\
        &+ q^5 \left(x^3 y^2-\frac{1}{x^3 y^2}-x^3 y+\frac{1}{x^3 y}+x^2 y^3-\frac{1}{x^2 y^3}-\frac{x^2}{y}+\frac{y}{x^2}-x y^3+\frac{1}{x y^3}+\frac{x}{y^2}-\frac{y^2}{x}\right) \\
        &+ 2q^6 \left( x^4 y^3-\frac{1}{x^4 y^3}- x^4 y+\frac{1}{x^4 y}+ x^3 y^4-\frac{1}{x^3 y^4}-\frac{x^3}{y}+\frac{y}{x^3}- x y^4+\frac{2}{x y^4}+\frac{x}{y^3}-\frac{ y^3}{x}\right) \\
        &+ q^7 \left(x^5 y^3-\frac{1}{x^5 y^3}-x^5 y^2+\frac{1}{x^5 y^2}+x^3 y^5-\frac{1}{x^3 y^5}-\frac{x^3}{y^2}+\frac{y^2}{x^3}-x^2 y^5+\frac{1}{x^2 y^5}+\frac{x^2}{y^3}-\frac{y^3}{x^2}\right) \\
        &- q^8 \left( x^5 y^5-\frac{1}{x^5 y^5}-x^5+\frac{1}{x^5}-y^5+\frac{1}{y^5}\right) + \cdots
    \end{split}
\end{equation*}
which matches with the theoretical computation for right-hand trefoil in \cite{Park20}.

\section{Remarks on higher rank}\label{sec:remarks-higher-rank}

The methods presented in this paper can be generalised to any $\sl_N$ by following a similar approach. Each finite-dimensional irreducible representation of $\sl_N$ is labeled by its highest weight; which is expressed in the fundamental weights basis as $(n_1,\dots,n_{N-1})$, where $n_i\in\N$. The irreducible representation $V_{(n_1,\dots,n_{N-1})}$ with the highest weight $(n_1,\dots,n_{N-1})$ is spanned by the vector with weight $(n_1,\dots,n_{N-1})$ of $V_{(n_1,0,\dots,0)}\otimes V_{(0,n_2,\dots,0)}\otimes\cdots\otimes V_{(0,0,\dots,n_{N-1})}$. The argument in Lemma \ref{lemma:RT-inv} extends directly to this case. For an analogue of Theorem \ref{thm:FK-SU(3)}, it would be necessary to explicitly compute the $R$-matrix associated to the tensor product of symmetric Verma modules to prove convergence of the quantum trace. This construction would give $F_K^{\sl_N}(x_1,\dots, x_{N-1},q)$ as defined in \cite{Park20} for positive braid knots.

In practice however formulas will rapidly get too involved. For general $\sl_N$, each crossing colored by the tensor product $V_{(n_1,0,\dots,0)}\otimes V_{(0,n_2,\dots,0)}\otimes\cdots\otimes V_{(0,0,\dots,n_{N-1})}$ is equivalent to $(N-1)^2$ crossings between the respective symmetric representations, whose $R$-matrix is dependent on $\frac{N(N-1)}{2}$ parameters, resulting in $\sim N^4$ parameters per crossing. This makes it hard to compute even the simplest knots by-hand. However, the setback is mainly computational, and we expect that more computational power or theoretical advances would allow for the methods used in \cite{gruen22} to extend to our situation as well.

\bibliographystyle{abbrv}
\bibliography{refs}

\end{document}